\documentclass[a4paper,12pt]{article}
\usepackage{localeng}
\usepackage{geometry}
\newtheorem*{lemma}{Lemma}
\newtheorem*{theorem}{Theorem}

\title{Complexity of majorants}
\author{Alexander Shen\thanks{LIRMM, CNRS \& Univesity of Montpellier, \protect\url{alexander.shen@lirmm.fr}. Supported by RaCAF ANR-15-CE40-0016-01 and RFBR 19-01-00563 grants.}}

\begin{document}

\maketitle

\begin{abstract}
The minimal Kolmogorov complexity of a total computable function that exceeds everywhere all total computable functions of complexity at most $n$, is $2^{n+O(1)}$. If ``everywhere'' is replaced by ``for all sufficiently large inputs'', the answer is $n+O(1)$.
\end{abstract}

The notion of Kolmogorov complexity of computable function was first considered by Schnorr~\cite{schnorr}. The (plain) complexity of a computable function is the minimal length of a program that computes this function. As usual, we require that the programming language is optimal, i.e., leads to a minimal complexity up to $O(1)$. One can also define the plain complexity of a function as the minimal \emph{complexity} of its programs. In this case we may use any G\"odel numbering of programs.

Consider all total computable functions that have complexity at most $n$. Alexey Milovanov asked the following question: \emph{What is a minimal complexity of a total computable function that exceeds all of them}? The words  ``$f$ exceeds $g$'' can be understood in different ways. Here are two possibilities:
\begin{itemize}
\item $f$ exceeds $g$ if $f(n)> g(n)$ for all $n$;
\item $f$ weakly exceeds $g$ if $f(n)> g(n)$ for all \emph{sufficiently large} $n$.
\end{itemize}

The following simple result answers both questions.

\begin{theorem}\label{thm:bounds}
\leavevmode
\begin{itemize}
\item The minimal complexity of a total computable function that weakly exceeds all total computable functions  of complexity at most $n$ is $n+O(1)$.
\item The minimal complexity of a total computable function that exceeds all total computable functions of complexity at most $n$ is $2^{n+O(1)}$.
\end{itemize}
\end{theorem}

\begin{proof}
The first part is easy. Obviously a function $g$ that weakly exceeds all total computable functions of complexity at most $n$ should have complexity greater than $n$ (since it does not exceed itself). On the other hand, assume that we know the number $T_n$ of programs of \emph{total} functions that have length at most~$n$. This number can be represented as an $(n+1)$-bit string. Knowing this string (and therefore knowing $n$), we can compute the following function~$g$:
\begin{quote}
\emph{to compute $g(k)$, enumerate programs that terminate at all inputs $0,1,\ldots,k$; as soon as $T_n$ programs with this property are found, output the maximal value of these programs on $k$, plus $1$.}
\end{quote}
The function $g$ is total since there are $T_n$ total programs. It weakly exceeds all total computable functions of complexity at most $n$. Indeed, if $k$ is large enough, the only programs of length at most $n$ that terminate on $0,1,\ldots, k$ are the total ones, so our computation process would discover exactly the total programs and return a number which exceeds all its values at $k$. Since $T_n$ is a string of length $n+O(1)$, the complexity of the function $g$ is at most $n+O(1)$. The first part is proven.

For the second part the upper bound is also easy. Indeed, let $T_n$ be the string of length $2^{n+1}$ that encodes information about all programs of length at most $n$ saying which of them are total. Knowing $T_n$ (and therefore knowing~$n$), we  compute $g(k)$ as the maximal value of all total programs of length at most $n$ for the input $k$, plus $1$. This function is total, exceeds all total computable functions of complexity at most $n$ (by construction), and has complexity at most $2^{n+O(1)}$.

The lower bound is a bit more difficult and it is convenient to use the game argument. Consider the following full information game with two positive integer parameters $a$ and $b$. Each of the two players, Alice and Bob, constructs sequences of natural numbers; Alice constructs $a$ sequences and Bob constructs $b$ sequences. Initially all the sequences are empty. At any moment Alice may extend any of her sequences by adding one more term (a natural number).  Bob may do the same for his sequences. The game is infinite; in the limit Alice has $a$ sequences of natural numbers (finite or infinite) and $B$ has $b$ sequences of natural numbers (finite or infinite). The winning player is determined by this limit as follows: \emph{Bob wins if one of his infinite sequences exceeds all infinite sequences of Alice}. (Finite sequences do not matter.)

Since the winner is determined by the limit position, the order of moves do not matter. One may assume, for example, that Alice and Bob make their moves in turn and each player may skip her/his turn. Indeed, player loses nothing if (s)he postpones the move.

Increasing $b$, we make the game easier for Bob, and increasing $a$ we make it easier for Alice. The following lemma solves this game.

\begin{lemma}\leavevmode
\begin{itemize}
\item If $b\ge 2^a$, Bob has a winning strategy.
\item If $b<2^a$, Alice has a winning strategy.
\end{itemize}
\end{lemma}

\begin{proof}[Proof of the lemma]
The first part of the lemma essentially repeats the argument for the upper bound, and is not used in the sequel. But the strategy is easy and it is instructive to compare it with the argument given above.

Alice constructs $a$ sequences indexed by $a$ labels. For each subset $X$ of labels (on the Alice's side) we allocate one Bob's sequence (label). Since $b\ge 2^a$, this is possible. The Bob's sequence that correspond to a set $X$ is constructed in a trivial way: its $k$th term is the maximum of $k$th terms of all Alice's sequences with labels in $X$, plus $1$; the length of Bob's sequence is the minimal length of Alice's sequences with labels in $X$. For the limit state, take the set $X$ of all labels of infinite Alice's sequences. The Bob's sequence that corresponds to $X$ is infinite and exceeds all of them.

The second part of the lemma can be proved by induction. For $a=1$ we need to consider the case $b=1$. Then Alice's strategy is to copy Bob's sequence: he has to construct an infinite sequence according to the game rules. Alice's sequence will be the same, and Bob's sequence does not exceed it.

Now the induction step. Imagine that the second claim is true for $a=k-1$ and all $b<2^{k-1}$ (equivalently, for $b=2^{k-1}-1$, since decreasing $b$ makes Bob's task only harder). Consider a winning strategy for Alice in this game. We want to use it to construct Alice's strategy for $a=k$ and $b<2^k$.  This new strategy will consist of two stages.

At the first stage Alice reserves one of $k$ sequences and does not touch it, so it remains empty, and Alice is in the situation where $a=k-1$. There is a winning strategy against Bob with $b=2^{k-1}-1$, and now Bob has more sequences. Still the strategy can be use \emph{mutatis mutandi} up to the moment when Bob has $2^{k-1}$ or more \emph{non-empty} sequences. Indeed, the labeling does not matter, so Alice may imagine that she plays against at most $2^{k-1}-1$ sequences, the non-empty ones. What happens if Alice applies her winning strategy (for $k-1$) in this setting? There are two possibilities.
\begin{itemize}
\item Bob never has $2^{k-1}$ or more non-empty sequences. Then Alice applies her strategy throughout the entire infinite game and wins (according to the induction assumption).
\item At some moment Bob creates $2^{k-1}$ or more non-empty sequences. Then Alice recalls that she has a reserved sequence that is still empty, and writes a large number as the first term of this reserved sequence. The number should be greater than all numbers that appear at the first places of the non-empty Bob sequences.

Then Alice extends this reserved sequence to an infinite sequence in an arbitrary way. More precisely, since Alice cannot add infinitely many terms in one move, this decision is essentially a commitment to add these terms one by one no matter what. After that, the $2^k$ non-empty sequences of Bob become useless since they do not exceed one of the Alice's infinite sequences.  Starting from this point, Alice ignores these useless Bob's sequences and plays with the remaining ones. There is less than $2^{k-1}$ of them, so Alice can use the induction assumption and win.

However, there is a problem: the winning strategy that exists by the induction assumption is for the initial position of the game, and now Alice already has put some numbers in her $k-1$ sequences (non-reserved ones). A priori this could prevent her from using the winning strategy. However, the properties of the game help: in this game, if Alice has a winning strategy for the initial position, she can win the game starting from any position. Indeed, let $N$ be the maximal length of all her sequences in this position. Alice may add arbitrary numbers to her sequences up to length $N$ and then forget about first $N$ terms in all sequences, and start the game anew. If she wins in this modified game, she wins in the original game, because the $N$ first terms could only make the Bob's task harder.

\end{itemize}
This argument finishes the proof of the lemma.
\end{proof}

It remains to explain how the Lemma is used to prove the theorem. This is a standard reasoning for game arguments in algorithmic information theory (see~\cite{muchnik}) and goes as follows. Consider the game for $b=2^a-1$ where $a=2^n$ and consider a ``blind'' Bob's strategy where he uses all programs of length less than $a$ (there are $2^a-1=b$ of them) to generate his sequences (applying each program to $0$, $1$,\ldots\ sequentially). Alice constructs $a=2^n$ sequences and wins. These $2^n$ Alice's sequences are computable, since both Alice and Bob strategies are computable. To compute each of these sequences it is enough to know its ordinal number that can be encoded by an $n$-bit string. Note the this string determines $n$ so we do not need to provide $n$ separately. Therefore, these $2^n$ sequences all have complexity $n+O(1)$, and no computable total function of complexity less than $a=2^n$ exceeds all infinite sequences (=total functions) among them. This finishes the proof of the lower bound.
\end{proof}

Author is grateful to all members of the LIRMM ESCAPE team (Montpellier) and  MSU/HSE Kolmogorov seminar (Moscow) who asked the question and discussed the answer.

\end{document}